\documentclass{amsart}
\usepackage{amsfonts,amssymb,amsmath,amsthm}
\usepackage{url}
\usepackage{enumerate}
\newtheorem{thm}{Theorem}[section]
\newtheorem{cor}[thm]{Corollary}
\newtheorem{lem}[thm]{Lemma}
\newtheorem{prop}[thm]{Proposition}
\theoremstyle{definition}
\newtheorem{defn}[thm]{Definition}
\theoremstyle{remark}
\newtheorem{rem}[thm]{ Remark}
\newtheorem{exe}[thm]{\bf Example}

\numberwithin{equation}{section}

\begin{document}
\title{Generalization of isomorphism theorems groups to Partial groups}

\author{Yahya N'dao and Adlene Ayadi}

 \address{Yahya N'dao, University of Moncton, Department of mathematics and statistics, Canada}
 \email{yahiandao@yahoo.fr }
\address{Adlene Ayadi, University of Gafsa, Faculty of sciences, Department of Mathematics,Gafsa, Tunisia.}

\email{adlenesoo@yahoo.com}

\thanks{This work is supported by the research unit: syst\`emes dynamiques et combinatoire:
99UR15-15} 

\begin{abstract}
In this paper, we define a new structure analogous to group,
called partial group. This structure concerns the partial
stability by the composition inner law. We generalize the three
isomorphism theorems for groups to partial groups.
\end{abstract}
\maketitle

\section{\bf Introduction }

A semigroup is an algebraic structure consisting of a set together
with an associative binary operation. A semigroup generalizes a
monoid in that a semigroup need not have an identity element. It
also (originally) generalized a group (a monoid with all inverses)
to a type where every element did not have to have an inverse,
thus the name semigroup(See \cite{CAH}, \cite{HJM}). Several other
algebraic structure generalize the notion of group as quasigroup,
hypergroup. In abstract algebra, a quasigroup is an algebraic
structure resembling a group in the sense that "division" is
always possible. Quasigroups differ from groups mainly in that
they need not be associative. A quasigroup with an identity
element is called a loop(see \cite{COH}, \cite{AMA}). In
\cite{MSK}, Kamran extended the notion of AG-groupoid to AG-group.
Later on the significant results on the topic were published in
\cite{QMS}. An AG-group is a generalization of abelian group and a
special case of quasigroup. The structure of AG-group is a very
interesting structure in which one has to play with brackets.
There is no commutativity or associativity in general.
\medskip

We are interested in this paper to introduce a lower structure
than those of groups structure, called \emph{partial group}. Most
of algebraic structures like groups, semigroups, quasigroup and
monoid  all contain an inner composition  law which is total (i.e.
the law acts on any element of a set $G$ and remains there). Order
our new structure is based on the partial action of this law.
There is a subset $E$ of $G$ which is stable by the law and it
admits the group structure. Firstly, the set G is chosen as a part
of a group $\Gamma$ and it is defined as a product (by the law of
$\Gamma$) of a subgroup $E$ of $\Gamma$ with a part D which is
free with $E$, that is to say $D$ is contained in a group
$\widetilde{D}$ satisfying $\Gamma = E.\widetilde{D}:=\{xd:\ \
x\in E,\ \ d\in \widetilde{D}\}$ and $\widetilde{D}\cap E= \{e\}$,
where $e$ is the identity element of $\Gamma$. We say that
$\widetilde{D}$ is supplement of $E$ in$\Gamma$. We define a
partial inner law $'.'$ in $G$ as follows: Let $a,b\in G$ and
write $a=xd$ and $b=yd'$ with $x,y\in E$ and $d,d'\in D$, one has:
$a. b=xy$. This is means that $$'.':\ G\times G \longrightarrow
E$$ $$\ \ \ \ \ \ \ \ \ \ \ \ \ \ \  (a, b)\longmapsto a.b=xy$$

For any $a=xd\in G$ with $x\in E$ and $d\in D$, denote by
$$Inv(a)=\{x^{-1}d:\ \ d\in D\}$$ the set of all inverses element
of $a$. An element of $Inv(a)$ is denoted by $a^{*}$.\ \\

 In a partial group, the definition of the identity element $e$
 differs from group as follow:
$$a.e=x \ \ \mathrm{for\  every} \ \ a=xd, \  x\in E\ \
\mathrm{and} \ \ d\in D$$ Moreover, we know that in a group every
element has an only inverse element (symmetric element) but for a
partial group we have "for every element $a$ of $G$, we have
$Inv(a)\neq\emptyset$".

 Secondly,  in the definition of subgroup, we replace the phrase  "every element of $H$ has an inverse"
 by  "every element $a\in H$, $Inv(a)$ meets $H$", and  we give by analogy, the definition of
partial subgroup $H$ of $G$ as follow: let $H\subset G$. Then $H$
is a \emph{partial subgroup} of $G$ if:\ \\ (i) $e\in H$\
\\
 (ii) $a.b \in H$ for every  $a,\ b \in H$ \ \\
 (iii) $Inv(a)\cap H \neq \emptyset$ for every $a\in H$.

\medskip

Define the right (resp. left) cosets of an element $a\in G$ by
$$aH=\{a.h:\ \ \ h\in H\},\ \ \ \ \ (\mathrm{resp}.\ \ Ha=\{h.a:\
\ \ h\in H\}).$$
 \ \\ \\ We define the normal partial subgroup
$N$ as follow:  $N$ is a normal partial subgroup, that is, $Na =
aN$ for all $a\in G$. The
 space $X$, quotient of $G$ by the normal partial subgroup $N$ is defined by the following  equivalence relation :
  $$a\ \sim_{N}\ b \ \ \ \Longleftrightarrow \\ \ \
\mathrm{there \ exists}\ \ h\in N\ \ \mathrm{such\ that} \ \
a.e=b.h$$
\medskip

A \emph{homomorphism} of partial group is a map $f: G_{1}\
\longrightarrow G_{2}$ between two partial groups $G_{1}$ and
$G_{2}$ and satisfying $f(a.b).e = f(a).f(b)$. In particular, if
$G_{1}=E_{1}D_{1}$ and  $G_{2}=E_{2}D_{2}$\ then $f(E_{1})\subset
E_{2}$ and $f_{/E_{1}}: E_{1}\longrightarrow E_{2}$ is a
homomorphism of groups.\ (see Lemma ~\ref{L:mm3}).
\\

By analogy, we will prove that the kernal $Ker(f)=\{a\in G_{1}:\ \
f(a)=e \}$ of $f$, is a normal partial subgroup of $G_{1}$.\ \\
Finally, we generalize the isomorphism's theorem for groups to
partial groups.

 Our principal results are
the following:
\medskip

\begin{thm}\label{t:1} (An image is a natural quotient). Let
$G_{1}=E_{1}D_{1}$, $G_{2}=E_{2}D_{2}$ be two partial groups and
$f : G_{1} \longrightarrow G_{2}$ be a partial group homomorphism.
Let its default kernel and the image be $K = f^{-1}(D_{2})$; $H
=\{f(a):\ \ a\in G_{1} \}$ ; respectively a normal partial
subgroup of $G_{1}$ and a partial subgroup of $G_{2}$. Then there
is a natural isomorphism $\widetilde{f} :\ G_{/K}\longrightarrow
H$; $aK \longmapsto f(a)$.
\end{thm}

\bigskip
\begin{thm}\label{t:2}  Let $G$ be a partial group. Let $H$ be a partial subgroup of $G$ and let $K$ be a normal
partial subgroup of $G$. Then there is a natural isomorphism of
partial group $$ HK_{/K} \longrightarrow H_{/(H \cap K)}$$
$$aK\longmapsto a(H \cap K)$$
\end{thm}

\bigskip

\begin{thm}\label{t:3} (Absorption property of quotients). Let $G$ be a partial group. Let $K$ be a
normal partial subgroup of $G$, and let $N$ be a subgroup of $K$
that is also a normal partial subgroup of $G$. Then $K_{/N}$ is a
normal partial subgroup of $G_{/N}$; and there is a natural
isomorphism $$(G_{/N})_{/(K_{/N})} \longrightarrow G_{/K}:\  \
aN.(K_{/N}) \longmapsto aK$$
\end{thm}

\bigskip

\section{{\bf Supplement subgroups and free subsets}}
\medskip

Let $\Gamma$ be a group. The composition law of $\Gamma$ is
denoted as simple multiplication. Let $E$ and $\widetilde{D}$ be
two subgroup of $\Gamma$. We say that $E$ and $\widetilde{D}$ are
\emph{supplement} in $\Gamma$ and denoted by $E_{\top}
\widetilde{D}=\Gamma$, if $E\cap \widetilde{D}=\{e\}$ and
$E.\widetilde{D}=\Gamma$, where $$E.\widetilde{D}=\{xd,\ \ x\in
E,\ d\in \widetilde{D}\}$$ and $e$ is the neutral element of
$\Gamma$. We say that $\Gamma$ is the internal direct product of
$E$ and $\widetilde{D}$. A subset $D\subset \Gamma$ is called
\emph{free} with $E$ in $\Gamma$ if $e\in D$ and there exists a
supplement $\widetilde{D}$ of $E$ in $\Gamma$ containing $D$.

\medskip

For example: \ \\ - if $\Gamma=\mathbb{Z}^{2}$,
$E=\{0\}\times\mathbb{Z}$ and
$\widetilde{D}=\mathbb{Z}\times\{0\}$ then
$E_{\top}\widetilde{D}=\Gamma$.\ \\ - if $\Gamma=\mathbb{Z}$,
$E=2\mathbb{Z}$ and $\widetilde{D}=3\mathbb{Z}$ then
$E_{\top}\widetilde{D}=\Gamma$.
\medskip

\begin{prop}\label{p:00101} Let $E$ be a subgroup of $\Gamma$ and $D$ be a free subset of $\Gamma$ with
$E$. Denote by $E.D:=\{xd,\ \ x\in E,\ d\in D\}$. Then:\
\\ (i) for every $a\in E.D$ there exist only one $x\in E$ and only
one $d\in D$ such that $a=xd$.\ \\ (ii) there exists an
equivalence relation in $E.D$ defined by: $$a\sim b\ \
\Longleftrightarrow \ \mathrm{there \ exists}\ \ x\in E\ \ \
\mathrm{such\ that}\ \ \ x^{-1}a,\ x^{-1}b\in D$$
\end{prop}
\medskip

\begin{proof} (i) Let $a\in E.\widetilde{D}$. If  $a=xd=x'd'$ with $x,x'\in E$ and
$d,d'\in\widetilde{D}$. Then $x^{-1}x=dd'^{-1}\in E\cap
\widetilde{D}$. It follows that $x=x'$ and $d=d'$ since
$E\cap\widetilde{D}=\{e\}$.\medskip

\ \\ (ii) - \emph{Reflexivity:} $a\sim a$ since if $a=xd$ with
$x\in E$ and $d\in \widetilde{D}$. Then
$x^{-1}a=x^{-1}a=d\in\widetilde{D}$.\ \\ - \emph{Symmetry:} $a\sim
b$ then $b\sim a$ (it is obvious).\ \\ - \emph{Transitivity:} If
$a\sim b$ and $b\sim c$ then there exist $x,y\in E$ such that
$x^{-1}a,\ x^{-1}b\in\widetilde{D}$ and $y^{-1}b,\
y^{-1}c\in\widetilde{D}$. It follows that $b=xd=yd'$ for some
$d,d'\in\widetilde{D}$. By (i) we have $x=y$. Hence,$x^{-1}a,\
x^{-1}c\in \widetilde{D}$ and so $a\sim c$. \ \\ We conclude that
$\sim$ is an equivalence relation.
\end{proof}
\medskip
\medskip

\section{{\bf Partial inner law and partial group}}
\medskip

Let $\Gamma$ be a group, $E$ be a subgroup of $\Gamma$ and
$D\subset \Gamma$ be a subset free with $E$ in $\Gamma$ and
containing the neutral element $e$.  Denote by $G:=E.D=\{xd,\ \
x\in E,\ d\in D\}$. We define a partial inner law $'.'$ in $G$ as
follows: Let $a,b\in G$ and write $a=xd$ and $b=yd'$ with $x,y\in
E$ and $d,d'\in D$: $$a. b=xy$$ This means that $$a. b
=\left\{\begin{array}{c}
  xy, \ \ \mathrm{if}\ \ a,b\notin E\cup D\ \ \ \ \ \ \ \  \ \  \\
  ab, \ \ \ \mathrm{if}\ \ a,b\in E \ \ \ \ \ \ \ \ \ \ \ \ \ \ \ \ \\
  a,\ \ \ \mathrm{if}\ \ a\in E \ \  \ \mathrm{and}\ \ \ b\in D \\
   b,\ \ \mathrm{if}\ \ b\in E \ \  \ \mathrm{and}\ \ \ a\in D \\
   e,\ \ \ \mathrm{if}\ \ a, b\in D \ \ \ \ \ \  \ \ \  \ \ \ \  \ \ \ \\
\end{array}\right.$$
\medskip

We obtain also $$a. e=\left\{\begin{array}{c}
  x, \ \ \ \mathrm{if}\ \ a\notin D \ \ \ \ \ \ \ \ \ \ \ \ \ \ \ \ \\
  e,\ \ \ \mathrm{if}\ \ a\in D \ \ \ \ \ \  \ \ \  \ \ \ \  \ \ \ \\
\end{array}\right.$$
\medskip

We can write also if $a=xd$ with $x\in E$ and $d\in D$ then
$$a=(a.e)d$$
\medskip

\begin{defn}\label{df:1} $(G,\ .)$ is a partial group with support the group $E$ and
defect the set $D\subset (G\backslash E)\cup\{e\}$
 (where $e$ is the neutral element of $E$), if:\ \\
(i) $G=E.D=\{xd,\ \ x\in E,\ d\in D\}$\ \\ (ii) $e\in
 D$\ \\ (iii) For every $a,b\in G$, one has $a.
b=xy$, where  $a=xd$ and $b=yd'$ with $x,y\in E$ and $d,d'\in D$.
\end{defn}
\medskip

In all the following, we denote by $G=E.D$ any partial group with
support $E$ and defect $D$.
\medskip

\begin{prop}\label{p:1} Let $(G,\  .)$ be a partial group with support the group $E$
and defect the set $D$. Then:\ \\ (i) $a.b\in E$ for every $a,b\in
G$
\
\\ (ii) $a.e=x$ for every $a=xd\in G$ \ \\ (iii)  $d.d'=e$ for
every $d,d'\in D$ \ \\ (iii) $(a.b).c=a.(b.c)=a.b.c$ for every
$a,b,c\in G$ (i.e. the partial law $(.)$ is associative)\ \\ (iv)
$x.d.d'=d.x.d'=d.d'.x=x$ for every $x\in E$, $d,d'\in D$\ \\ (v)
$a.e=a$ if and only if $a\in E$
\end{prop}
\medskip

\begin{proof} Let $a,b,c\in G$ and write $a=xd,\ b=yd'$ and $c=zd"$
with $x,y,z\in E$ and $d,d',d"\in D$. \ \\ \\ Assertion (iii): We
have $(a.b).c=(xy).zd"=xyz$, $a.(b.c)=xd.(yz)=(xd).(yze)=xyz$ and
$a.b.c =(xd).(yd').(zd")=xyz$.\ \\ \\ Assertion (i) follows
directly from the definition.\ \\ \\ Assertion (ii): We have
$a.e=(xd).(ee)=xe=x$. \ \\ \\ Assertion (iv): By (iii), we have
$x.d.d'=(xe).(ed).(ed')=xee=x$, $d.x.d'=(ed).(xe).(ed')=exe=x$ and
$d.d'.x=(ed).(ed').(xe)=eex=x$. \ \\ \\ Assertion (v):  If $a.e=a$
then $a\in E$ by definition. Conversely, if $a\in E$ then $a=ae$
then $a.e=(ae).(ee)=ae=a$.
\end{proof}
\medskip

\begin{prop}\label{p:2} Let $(G,\ .)$ be a partial group with support the group $E$ and
 defect the set $D$. Then:\ \\ (i) $a.b=e$ if and only if
$y=x^{-1}$, for every $a=xd, b=yd'\in G$ ($x,y\in E$ and $d,d'\in
D$)\ \\ (ii) for every $a=xd\in G$ the set of all  reverse points
of $a$ is $Inv(a):=\{x^{-1}d',\ d'\in D\}$. \
\\ (iii) $\underset{n-times}{\underbrace{a.\dots. a}}=x^n$ for every $a=xd\in G$,
($x\in E$ and $d\in D$)
\end{prop}
\medskip

\begin{proof} Let $a,b,c\in G$ and write $a=xd,\ b=yd'$ and $c=zd"$
with $x,y,z\in E$ and $d,d',d"\in D$. \ \\  (i) We have
$a.b=(xd).(yd')=xy=e$ so $y=x^{-1}$. Conversely, if $y=x^{-1}$
then $a.b=(xd).(x^{-1}d')=xx^{-1}=e$.\ \\ (ii) Let $b\in G$. The
proof results directly from (i) because one has $a.b=e$ if and
only if $b=x^{-1}d'$ with $d'\in D$.\ \\ (iii) By
Proposition~\ref{p:1}.(iii), we have
$\underset{n-times}{\underbrace{a.\dots.
a}}=\underset{n-times}{\underbrace{(xd).\dots.
(xd)}}=\underset{n-times}{\underbrace{x\dots x}}=x^n$.
\end{proof}
\medskip

\begin{exe}  We give two trivial examples:\ \\ (i) Every group $E$ is a partial group with support
$E$ and defect $D=\{e\}$\ \\ (ii) Every subset $D$ of a group
$\Gamma$, containing the identity element $e$ of $\Gamma$, is a
partial group with support $E=\{e\}$ and  defect $D$.
\end{exe}
\medskip

\begin{prop}\label{p:3} Let $H$ be a group, $K$ and $L$ are two subgroups of
$H$. Suppose that $L$ is free with $K$ in $H$. Then: \ \\ (i) The
set $G:=K.L=\{xd,\ x\in K,\ d\in L\}$ is a partial group with
support $K$ and defect $L$. We say that $G$ is \emph{melted} in
$K$. \ \\ (ii) The set $G':=\{e\}.K=\{ed,\ d\in K\}$ is a partial
group with support $\{e\}$ and defect $K$. We say that $K$ is
\emph{totally melted} in $\{e\}$.
\end{prop}
\medskip

\begin{proof} The proof follows directly from the definition.
\end{proof}
\medskip

\begin{prop}\label{p:aa01} The quotient space $X$ of $E.D$ by the
above equivalence relation is a group isomorphic to $E$.
\end{prop}
\medskip

\begin{proof} Denote by $\overline{a}=\{b\in E.D:\ \ a\sim b\}$ the
equivalence classes of a point $a\in E.D$. Write $a=x.d$ with
$x\in E$ and $d\in D$. See that $\overline{a}=\overline{x}$. This
means that $X=(E.D)_{/\sim}=E_{/\sim}$. Define the inner law '$*$'
on $X$ by: $$\overline{a}*\overline{b}=\overline{a.b}$$ We have
$(X,\ *)$ is a group; Indeed:\ \\ - '$*$' is an inner law in $X$.\
\\ - '$*$' has $\overline{e}$ as the neutrally element, since
$\overline{e}*\overline{a}=\overline{a}*\overline{e}=\overline{a.e}=\overline{x}=\overline{a}$
for every $a=xd\in E.D$.\ \\ - The inverse of any element
$\overline{a}=\overline{x}$ is $\overline{x^{-1}}$, since
$\overline{x}*\overline{x^{-1}}=\overline{x^{-1}}*\overline{x}=\overline{x.x^{-1}}=\overline{e}.$\
\\ - '$*$' is associative since is the law '$.$' of $E$.\ \\ \\
Denote by $\pi: E \longrightarrow X$ be the canonical projection
given by $\pi(x)=\overline{x}$. The projection $\pi$ is a
homomorphism of groups by construction.\ \\ Now, $\pi$ is
injective: $\pi(x)=\overline{e}$ implies that
$\overline{x}=\overline{e}$ so $x\in D\cap E$. As $E$ and $D$ are
free, $x=e$.\ \\ $\pi$ is surjective by construction. It follows
that $\pi$ is an isomorphism of group.
\end{proof}

\medskip

\section{{\bf Partial subgroup of a partial group}}
In words, this tells us that for a subset to be a partial
subgroup, it must be nonempty and contain the products and
inverses of all its elements.

\begin{defn} Let $G$ be a partial group with partial neurtral element $e$ and $H\subset G$. Then $H$ is a \emph{partial subgroup} of $G$ if:\ \\
(i) $e\in H$\ \\
 (ii) $a.b \in H$ for every  $a,\ b \in H$ \ \\
 (iii) $Inv(h)\cap H \neq \emptyset$ for every $h\in H$
\end{defn}
\medskip

\begin{rem} If $G$ is a group then it is a partial group with
defect $\{e\}$. Moreover, any subgroup $H$ of $G$ is also a
partial subgroup of $G$ with defect $\{e\}$.
\end{rem}
\medskip

\begin{prop}\label{p:4} Let $H$ be a partial subgroup of the
partial group $G$ with support $E$ and defect $D$. Then $H=F.D'$
with $F$ is a subgroup of $E$ and defect $D'\subset D$.
\end{prop}
\medskip

\begin{proof} By definition we have  $Inv(h)\cap H\neq\emptyset$ for every $h\in
H$, in particular $Inv(e)\cap H\neq\emptyset$. Denote by
$D'=Inv(e)\cap H$, so $e\in D'$. Since $Inv(e)=\{e^{-1}d,\ \ d\in
D \}=D$ then $D'= D\cap H\subset D$. Now, denote by $$F=\{x\in E,\
\ \mathrm{ for\ which }\ \mathrm{there\ is} \ d\in D \ \
\mathrm{such \ that}\ xd\in H\}.$$ We can verify that $F=E\cap H$
and that $F$ is a subgroup of $E$.
\end{proof}
\medskip

\begin{cor}\label{C:1} Let $G$ be a partial group with support $E$ and defect $D$ and $H$ be a
partial subgroup of $G$ with support $F$ and defect $D'$. Then:\
\\ (i) $F=E\cap H$  \ \\ (ii) $D'=D\cap H$
\end{cor}
\medskip

\begin{proof} We have $G=E.D$ and $H=F.D'$.\ \\ (i) By
Proposition~\ref{p:4}, $F\subset E$ so $F\subset E\cap H$. For the
converse, let $a\in E\cap H$. Then $a=xd$ with $x\in F$ and $d\in
D'$. By Proposition~\ref{p:4}, $D'\subset D$. Since $D$ is free
with $E$ so is $D'$. Since $a\in E$ then $d=e$ and so $a=x\in F$.\
\\ (ii) By
Proposition~\ref{p:4}, $D'\subset D$ so $D'\subset D\cap H$. For
the converse, let $a\in D'\cap H$. Then $a=xd$ with $x\in F$ and
$d\in D'$. By Proposition~\ref{p:4}, $F\subset E$. Since $D$ is
free with $E$ so is $D'$. Since $a\in D'$ then $x=e$ and so
$a=d\in D'$.
\end{proof}
\medskip

\begin{cor}\label{C:2} Any partial subgroup of a partial group $G$ is
a partial group with the restriction of the partial inner law of
$G$ to $H$.
\end{cor}
\medskip

\begin{proof}
\end{proof}
\medskip

\begin{defn} Let $G$ be a partial group with support $E$ and defect $D$ and $H$ be a partial subgroup of $G$ with support $F$ and defect $D'$.
\ \\ (i) We say that $H$ has a total  support if $E=F$\ \\ (ii) We
say that $H$ has a total  defect if $D'=D$
\end{defn}

\medskip

\begin{prop}\label{p:5} Let $G$ with support $E$ and defect $D$ and $H$ be a
partial subgroup of $G$ with support $F$ and defect $D'$. Then:\
\\ (i) $H$ has a  total defect if and only if $Inv(h)\subset H$
for every $h\in H$\ \\ (ii) $H$ has a  total support if and only
if for every $x\in E$ there is $d\in D'$ such that $xd\in H$.
\end{prop}
\medskip

\begin{proof} (i) If $H$ has a total defect then $D'=D$, so for every $h=xd\in H$
with $x\in F$ and $d\in D$, we have $Inv(h)=\{x^{-1}d',\ \ d'\in D
\}\subset H$. Conversely, since $Inv(e)=\{e^{-1}d,\ \ d\in D \}=D$
then $ D\subset H$, so by Corollary ~\ref{C:1}, $D'= D\cap H=D$.
\
\\ (ii) The directly sense is obvious by definition. For the
converse, let $x\in E$ then there exists $d\in D'$ such that
$a:=xd\in H$, so $a.e=x\in F$.
\end{proof}
\bigskip

\section{{\bf Partial Cosets}}

\medskip

Given any partial subgroup $H$ of a partial group $G$, we can
construct a partition of $G$ into \emph{partial cosets} of $H$,
just as we did for rings. But for partial groups, things are a bit
more complicated. Because the partial group operation may not be
commutative, we have to define two different sorts of partial
cosets. Write $H=F.D'$ and define a relation $\sim_{r}$ on $G$ by
the rule $a \sim_{r} b$ if and only if $b.a^{*}\in F$ for some
$a^{*}\in Inv(a)$. We claim that $\sim_{r}$ is an equivalence
relation:
\
\\ - reflexive: For any $a \in G$, we have $a.a^{*} =e\in F$, so
$a \sim_{r} a$.\ \\ - symmetric: Suppose that $a \sim_{r} b$, with
$a=xd$ and $b=yd'$, so that $h = b.a^{*}=yx^{-1}\in F$. Then
$h^{-1} =xy^{-1} = a.b^{*} \in F$, so $b \sim_{r} a$. \ \\ -
transitive: Suppose that $a\sim_{r} b$ and $b \sim_{r} c$, with
$a=xd$, $b=yd'$ and $c=zd''$, so that $h=b.a^{*}=yx^{-1}\in F$ and
$k =c.b^{*}=zy^{-1}\in F$. Then $kh = (zy^{-1})(yx^{-1}) =
zx^{-1}=c.a^{*}\in F$, so $a \sim_{r} c$. \ \\ \\

The equivalence classes of this equivalence relation are called
the \emph{right partial cosets} of $H$ in $G$. A right partial
coset is a set of elements of the form $Ha = \{h.a :\ \  h \in
H\}$, for some fixed element $a\in G$ called the partial coset
representative. For

$$b \in Ha,\ \ \ \Longleftrightarrow  \ \  \ b.e = h.a\  for\ some
\ \ \ h \in H \ \ \ \Longleftrightarrow \ \ b.a^{*}\in F\ \ \ \
\Longleftrightarrow \ \ \ \ a \sim_{r}b.$$
\medskip

We summaries all this as follows:

\medskip

\begin{prop} If $H$ is a partial subgroup of the partial group $G=E.D$, then $E$ is
partitioned into right partial cosets of $H$ in $G$, sets of the
form $Ha = \{h.a : \ \ \ h \in H\}$.
\end{prop}
\medskip

In a similar way, the relation $\sim_{l}$ defined on $G$ by the
rule $a \sim_{l} b$ if and only if $a^{*}.b\in F$ is an
equivalence relation on $G$, and its equivalence classes are the
left partial cosets of $H=F.D'$ in $G$, the sets of the form $$aH
= \{a.h : \ \ h \in H\}.$$

 If $G$ is an abelian partial group, the left and
right partial cosets of any partial subgroup coincide, since $Ha =
\{h.a :\ \  h \in H\} = \{a.h : \ \ h \in H\} = aH$.

\medskip

We say that $G$ is abelian if $a.b=b.a$ for every $a, b\in G$.
\medskip

\section{{\bf Normal partial subgroups}}

 A \emph{normal} partial subgroup is a special kind of partial subgroup
of a partial group. Recall from the last chapter that any partial
subgroup $H=F.D'$ has right and left cosets, which may not be the
same. We say that $H$ is a \emph{normal subgroup} of $G$ if the
right and left cosets of $H$ in $G$ are the same; that is, if $Ha
= aH$ for any $a \in G$. There are several equivalent ways of
saying the same thing. We define $$a^{*}Ha = \{a^{*}.h.x :\ \ \ \
h \in H\}$$ for any element $a\in G$ and $a^{*}\in Inv(a)$. We can
show that $a^{*}Ha=x^{-1}Fx$ and $a=xd$.
\medskip

\begin{prop}\label{p:121001} Let $H=F.D'$ be a partial subgroup of the partial group $G=E.D$. Then the following are
equivalent: \ \\ (a) $H$ is a normal partial subgroup, that is,
$Ha = aH$ for all $a\in G$.\ \\ (b) $a^{*}Ha = F$ for all $a\in
G$.\ \\ (c) $a^{*}.h.a\in F$, for all $a\in G$ and $h \in H$.\ \\
(d) $x^{-1}Fx=F$ for all $x\in E$.
\end{prop}
\medskip

\begin{proof} If $Ha = aH$, then $a^{*}Ha = a^{*}.aH =eH=F$, and conversely. So (a) and (b) are
equivalent. If (b) holds then every element $a^{*}.h.a$ belongs to
$a^{*}Ha$, and so to $H$, so (c) holds. Conversely, suppose that
(c) holds. Then every element of $a^{*}Ha$ belongs to $H$, and we
have to prove the reverse inclusion. So take $h\in H$. Putting $b
= a^{*}$, we have $c = b^{*}.h.b = a.h.a^{*}\in F$, so $h.e\in
a^{*}Ha$, finishing the proof. Now the important thing about
normal patial subgroups is that, like ideals, they are kernels of
homomorphisms.
\end{proof}
\medskip

\begin{prop}\label{p:144101}  If $G$ is Abelian, then every partial subgroup $H$ of
$G$ is normal.
\end{prop}
\medskip

\begin{proof} The proof is obvious because, if $G$ is Abelian, then $aH = Ha$ for all $a\in G$.
\end{proof}
\medskip

For the record, here is a normal subgroup test:

\begin{prop}\label{p:1dd101} (Normal partial subgroup test) A non-empty subset $H=F.D'$ of a partial group $G=E.D$ is a
normal partial subgroup of $G$ if the following hold: \ \\ (a) for
any $h,k \in H$, we have $h.k^{*} \in F$.\ \\ (b) for any $h\in H$
and $a\in G$, we have $a^{*}.h.a\in F$.
\end{prop}
\medskip

 \begin{proof} (a) is the condition
of the second partial subgroup Test, and we saw that (b) is a
condition for a partial subgroup to be normal.
\end{proof}
\medskip

\section{{\bf Quotient of partial group}}
\medskip

Let $H=F.D'$ be a normal partial subgroup of a partial group
$G=E.D$. We define the quotient group $G_{/H}$ as follows: \ \\
$\diamond$  The elements of $G_{/H}$ are the cosets of $H$ in $G$
(left or right doesn't matter, since $H$ is normal).\ \\
$\diamond$ The group operation is defined by $(Ha)(Hb) = Hab$ for
all $a,b\in G$; in other words, to multiply cosets, we multiply
their representatives.

\medskip

\begin{defn} If $G$ is a partial group and $H \subset G$ is a partial subgroup, define a relation on
$G$ by $$a\ \sim_{H}\ b \ \ \ \Longleftrightarrow \\ \ \
\mathrm{there \ exists}\ \ h\in H\ \ \mathrm{such\ that} \ \
a.e=b.h$$
 We say that $a$ is congruent to $b$ mod $H$.
\end{defn}
\medskip

Now, write $H=F.D'$ then  saying that $a.e=b.h$ for some $h \in H$
is the same as saying that there exists $a^{*}\in Inv(a)$ such
that $a^{*}.b\in F$. Thus:
\medskip

\begin{lem} Congruence mod $H$ is an equivalence relation.
\end{lem}

 What do the equivalence classes look like?
\medskip

\begin{defn} If $H$ is apartial subgroup of the partial group $G$ and $g\in G$,
 then define the left coset of $H$ containing $g$ as $gH =\{g.h \ |\  h \in H
\}$.
 \end{defn}
\medskip

\begin{prop}(Congruence mod H and Cosets). Let $H=F.D'$ be a partial subgroup of the partial group $G=E.D$. Then\ \\ (a) The following
are equivalent:
\begin{itemize}
  \item [(i)] $a\ \sim_{H}\ b$
  \item [(ii)] $a^{*}b \in F$ for every $a^{*}\in Inv(a)$
  \item [(iii)] $b.e = a.h$ for some $h\in H$
  \item[(iv)] $b.e \in aH$
  \item[(v)] $bH \subset aH$
  \item[(vi)] $bH = aH$
\end{itemize}\ \\ (b) The left cosets $aH$ are the equivalence classes of
equivalence mod $H$. Thus, two left cosets are either equal or
disjoint (this being true of equivalence classes in general).
\end{prop}

\medskip

\begin{proof}  By definition, we prove (a). For
part (b), denote the equivalence class of $a \in G$ by $[a]$. One
has \begin{align*} b \in [a]\ & \Longleftrightarrow  \ a\
\sim_{H}\ b \ \ \ \ (\mathrm{Definition\  of \ equivalence\
classes})\\ \ & \Longleftrightarrow \ b.e \in aH \ &  \
\end{align*}
 By part (a), (i)$\Longrightarrow(iv)$ whence $[a] = aH$. That is, the equivalence classes are just
the left cosets, as required.
\end{proof}
\medskip

\begin{prop}\label{p:12121xvv} If $N$ is a normal partial subgroup of $G$, then the quotient group
$G_{/N}$ as defined above is a group. Moreover, the map $\pi$ from
$G$ to $G_{/N}$ defined by $\pi(a) = Na$ is a homomorphism whose
kernel is $N$ and whose image is $G_{/N}$.
\end{prop}
\medskip

 \begin{proof} First we have to show that the
definition of the group operation is a good one. In other words,
suppose that we chose different coset representatives $a'$ and
$b'$ for the cosets $Ha$ and $Nb$; is it true that $Na.b =
Ha'.b'$? We have $x' = h.x$ and $y' = k.y$, for some $h,k \in N$.
Now $a.k$ belongs to the left coset $aN$. Since $N$ is normal,
this is equal to the right coset $Ha$, so that $a.k = l.a$ for
some $l\in N$. Then $a'b' = h.a.k.b = (h.l).(a.b)\in Na.b$, since
$h.l \in N$. Thus the operation is indeed well defined.
\end{proof}
\medskip

\begin{lem}\label{L:111414} (The Quotient Group). If $N$ is a normal partial subgroup of the partial group $G$,
 then the multiplication of left cosets turns $G_{/N}$ into a group, called the quotient partial group.
\end{lem}
\medskip

\begin{proof} We just check the axioms.
\end{proof}
\medskip

\section{{\bf Homomorphisms of partial groups}}

An isomorphism between partial groups has two properties: it is a
bijection; and it preserves the partial group operation. If we
relax the first property but keep the second, we obtain a
homomorphism. Let $G_{1}=E_{1}.D_{1}$ and $G_{2}=E_{2}.D_{2}$ two
partial groups. Just as for rings, let $f : G_{1}\longrightarrow
G_{2}$ such that $f(E_{1})\subset E_{2}$ and $f(D_{1})\subset
D_{2}$. We say that a function $f$ is :\ \\ \\ $\diamond$ a
\emph{homomorphism} of partial group if it satisfies $$f(g.h).e =
f(g).f(h)\ \ \ \ \ \ \ (1)$$\
\\  $\diamond$ a \emph{monomorphism} of partial group if it
satisfies (1) and is one-to-one;\ \\ \\ $\diamond$ an
\emph{epimorphism} of partial group if it satisfies (1) and is
onto.\
\\ \\ $\diamond$ an \emph{isomorphism} of partial group if it satisfies (1) and is
one-to-one and onto.
\medskip

We have the following lemma, proved in much the same way as for
rings:

\medskip

\section{{\bf Proof of Theorems ~\ref{t:1}, ~\ref{t:2} and ~\ref{t:3}}}
\medskip

\begin{lem} (Inverse of a Homomorphism). Let $G_{1}=E_{1}D_{1}$ and $G_{2}=E_{2}D_{2}$ be two partial group. If $f : G_{1} \longrightarrow G_{2}$ is an
homomorphism of partial group, bijective with
$f^{-1}(E_{2})\subset E_{2}$ and $f^{-1}(D_{2})\subset D_{1}$,
then $f^{-1} : G_{2} \longrightarrow G_{1}$ is also a partial
group homomorphism.
\end{lem}
\medskip

\begin{proof} All we need to show is that $f^{-1}(a.b)= f^{-1}(a).f^{-1}(b)$
 for every $a, b \in G_{2}$. But, since $f$ is a homomorphism, we
 have
 $f(f^{-1}(a).f^{-1}(b)).e_{2}=f(f^{-1}(a)).f(f^{-1}(b))=a.b$.
 Then there exists $d\in D_{2}$ such that
 $f(f^{-1}(a).f^{-1}(b))=(a.b)d$. As $f^{-1}(a).f^{-1}(b)\in E_{1}$ then  $f(f^{-1}(a).f^{-1}(b))\in f(E_{1})\subset E_{2}$, so
 $(a.b)d\in E_{2}$, hence $d= e_{1}$. It follows that $f(f^{-1}(a).f^{-1}(b))=a.b$. Since $f$ is bijective then
 $f^{-1}(a).f^{-1}(b)=f(a.b)$.
\end{proof}
\medskip

\begin{lem} Let $f : G_{1} \longrightarrow G_{2}$ be a homomorphism of partial group. Then $f(e) = e$; $f(a^{*})
\in Inv(f(a))$;  and $f(a.b^{*}).e = (f(a)).(f(b))^{*}$, for all
$a,b \in G_{1}$.
\end{lem}

 Now, if $f : G_{1} \longrightarrow
G_{2}$ is a homomorphism, we define the image of $f$ to be the
subset $$Im(f):=\{b \in G_{2} :\ \ b = f(a)\ \mathrm{for\ some} \
a\in G_{1}\}\ \ of \ \ G_{2},$$ and the kernel of $f$ to be the
subset $$Ker(f):=\{a\in G_{1} :\ \ \  f(a) = e\}\ \ \ \mathrm{of}\
\ G_{1}.$$

\medskip

\begin{lem}\label{L:mm3} Let $G_{1}=E_{1}.D_{1}$ and $G_{2}=E_{2}.D_{2}$ be two partial groups and
 $f : G_{1} \longrightarrow G_{2}$ be an homomorphism of partial group. Then:\ \\ (i) $f_{/E_{1}}: E_{1}\longrightarrow E_{2}$
 is a homomorphism of group.\ \\ (ii)
 $f(e_{1})=e_{2}$.\ \\ (iii) $f(a^{*})\in Inv(f(a))$ for every $a\in
 G_{1}$.\ \\ (iv) for every $a=xd\in G_{1}$, with $x\in E_{1}$ and
 $d \in D_{1}$, one has $f(a)=f(x)d''$, for some $d''\in D_{2}$.
\end{lem}
\medskip

\begin{proof} (i) $f(xy)= f(x.y).e_{2}=f(x).f(y)=f(x)f(y)$ for every $x,y\in E_{1}$.\ \\
(ii)
$f(e_{1})=f(x.x^{-1}).e_{2}=f(x).f(x^{-1})=f(x)(f(x))^{-1}=e_{2}$.\
\\ (iii) Let $a\in G_{1}$. We have
$e_{2}=f(e_{1})=f(a.a^{*})=f(a.a^{*}).e_{2}=f(a).f(a^{*})$, so
$f(a^{*})\in Inv(f(a^{*}))$.\ \\ (iv) We have $f(a.e_{1})=f(x)\in
E_{2}$, so $f(a.e_{1})=f(a.e_{1}).e_{2}=f(a).f(e_{1})=f(a).e_{2}$.
Hence $f(a)=f(x).d''$ for some $d''\in D_{2}$.
\end{proof}
\medskip

Define the default kernel of $f$
$\widetilde{Ker(f)}=f^{-1}(D_{2})$.

We have $Ker(f)\subset \widetilde{Ker(f)}$ since $e_{2}\in D_{2}$,
so $f^{-1}(e_{2})\subset f^{-1}(D_{2})$.

\medskip

\begin{prop}\label{p:06543} Let $f : G_{1} \longrightarrow G_{2}$ be an
isomorphism of partial group. Then the default kernel
$\widetilde{Ker(f)}$ and the image $Im(f)$ are respectively normal
partial subgroup  of $G_{1}$ and partial subgroup of $G_{2}$.
\end{prop}
\medskip

\begin{proof} $\diamond$ $\widetilde{Ker(f)}$ is a partial subgroup of $G_{1}$: \ \\ - By Lemma
~\ref{L:mm3},(i), we have $e_{1}\in \widetilde{Ker(f)}$.\ \\ -
Take $a,b \in \widetilde{Ker(f)}$. Then $f(a), f(b)\in D_{2}$.
Write $f(a)=d$ and $f(b)=d'$, so  $f(a.b).e_{2}
=f(a).f(b)=d.d'=e_{2}$. Then $f(a.b)\in D_{2}$. \ \\ - Let $a\in
\widetilde{Ker(f)}$ and $a^{*}\in Inv(a)$, so $f(a^{*})\in
Inv(f(a))$ (By Lemma~\ref{L:mm3},(ii)). Write $f(a)=d\in D_{2}$
then by definition $Inv(f(a))=Inv(d)=\{d' :\ \ d'\in
D_{2}\}=D_{2}$, so
  $a^{*}\in f^{-1}(D_{2})$. It
follows that $Inv(a)\subset \widetilde{Ker(f)}$.

We conclude that $\widetilde{Ker(f)}$ is a partial subgroup of
$G_{1}$.\ \\

 Suppose that $b\in \widetilde{Ker(f)}$ and $a\in G$. Write $f(b)=d\in D_{2}$. Then $f(a^{*}.b.a)=
f(a^{*})·f(b).f(a) = (f(a))^{*}.d. f(a) = (f(a))^{*}.f(a)=e_{2}$,
so $a^{*}.b.a\in ker(f) \subset \widetilde{Ker(f)}$. Hence,
$\widetilde{Ker(f)}$ is a normal subgroup of $G_{1}$.\ \\
\\ $\diamond$ $Im(f)$ is a normal partial subgroup of
$G_{2}$: \
\\ - By Lemma ~\ref{L:mm3},(i), we have $f(e_{1})=e_{2}\in Im(f)$.\
\\ - Take $a,b \in Im(f)$. Write $a=f(x).d$ and
$b=f(y).d'$ for some $x,y\in E_{1}$ and $d,d'\in D_{2}$. Since
$f(E_{1})\subset E_{2}$ then $f(x), f(y)\in E_{2}$, so
$a.b=f(x)f(y)=f(xy)\in Im(f)$.
\
\\ - Let $a\in Im(f)$ and $a^{*}\in Inv(a)$. Write $a=f(x)d$ with $x\in
E_{1}$ and $d\in D_{2}$. By definition, $Inv(a)=\{f(x)^{-1}d':\ \
d'\in D_{2} \}$, so $a^{*}=f(x)^{-1}d'$ for some $d'\in D_{2}$. By
Lemma~\ref{L:mm3}.(i), $f(x)^{-1}=f(x^{-1})\in Im(f)$ then
$a^{*}\in Im(f)$, so $Inv(a)\subset Im(f)$.

We conclude that $Im(f)$ is a partial subgroup of $G_{2}$.\ \
\end{proof}
\medskip

\begin{prop}\label{Lad:011} Let $H=F.D'$ and $K=F'.D''$ two
partial subgroups of $G=E.D$. Then:\ \\ (i) $H.K=F.F'$.\ \\ (ii)
$aH=xF$  and $Ha=Fx$ for every $a=xd\in G$ and $x\in E$.\ \\ (iii)
$H$ is normal in $G$ if and only if $F$ is normal in $E$.\ \\ (iv)
if $H$ is normal then $G_{/H}=E_{/F}$.\ \\ (v) $H\cap K=(F\cap
F').(D'\cap D'').$\ \\ (vi) For every $a\in H$, $aH=e.H=Ha$, where
$e$ is the identity element of $H$.
\end{prop}
\medskip

\begin{proof}(i) We have $H.K=\{a.b :\ \ a\in H,\ b\in K\}$. If $a=xd\in H$ and $b=yd'\in K$ with $x\in F$ and $y\in F'$, then
 $a.b=xy$, so
$H.K=\{x.y :\ \ x\in F,\ y\in F'\}$.\ \\ (ii) Write $a=xd$ with
$x\in F$ and $d\in D$ then $$aH=\{a.b :\ \ \ b\in K\}=\{xy :\ \ \
y\in F'\}=xF.$$ It follows then $Ha=Fx$ with the same proof.\ \\
(iii) Suppose that $H$ is normal in $G$ then $aH=Ha$ for every
$a=xd\in \in G$. By (i) we have $aH=xF$ and $Ha=Fx$ so $xF=Fx$ foe
every $x\in E$. Hence, $F$ is normal subgroup in $E$. We use the
same proof for the converse. \ \\ (iv) If $H$ is normal partial
subgroup in $G$ then by (iii), $F$ is normal subgroup in $E$. Let
$[a]\in G_{/H}$ with $a=xd\in G$ and $[a]=aH$. By (ii), $aH=xF$,
so $[a]\in E_{/F}$. The same proof is used for the converse.\ \\
(v) It is clear that $(F\cap F').(D'\cap D'')\subset H\cap K$. For
the converse, let $a\in H\cap K$. Then there exists $x\in F$,
$y\in F'$, $d\in D'$ and $d'\in D''$ such that $a=xd=yd'$. By
Proposition ~\ref{p:4}, we have $D'\cup D''\subset D$ and $F\cup
F'\subset E$. By proposition ~\ref{p:00101}, $x=y$ and $d=d'$. It
follows that $x\in F\cap F'$ and $y\in D'\cap D''$.\ \\ (vi) Let
$a\in K$ and write $a=xd$ with $x\in F$ and $d\in D'$, then
$aH=\{a.b:\ \ b\in H\}=\{xy:\ \ y\in F\}=xF=F$. In the same way we
have $F=Fx=Ha$.
\end{proof}
\medskip

\begin{proof}[Proof of Theorem~\ref{t:1}]  By Proposition ~\ref{p:06543}, the default kernel and default image of $f$ be
$K = \widetilde{ker(f)}$; $H = Im(f)$; respectively a normal
partial subgroup of $G_{1}$ and a patial subgroup of $G_{2}$.
Write $G_{1}=E_{1}D_{1}$, $G=E_{2}D_{2}$, $K=F_{1}D'_{1}$ and
$H=F_{2}D'_{2}$. Then there is a natural isomorphism
$\widetilde{f} : G_{/K}\ \longrightarrow H$; given by
$aK\longmapsto f(a)$. Denote by $\widetilde{e}=eK$. The map
$\widetilde{f}$ is well defined because if $a'K = aK$ then $a'.e =
a.k$ for some $k \in K$ and so

$$f(a').e  = f(a.k)= f(a).f(k) = f(a).\widetilde{e}= f(a)$$
\medskip

  The map $\widetilde{f}$ is a homomorphism of partial group because $f$ is a homomorphism of partial group,

\begin{align*}
   \widetilde{f}(aK.a'K).e_{2} &  = \widetilde{f}((a.a')K).e_{2} \ \ \ \ \mathrm{by\ definition \ of \ coset \ multiplication} \\
  \ &  = f(a.a').e_{2} \ \ \ \ \mathrm{by\ definition \ of} \ \widetilde{f}\\
  \ &  = f(a).f(a') \ \ \ \ \mathrm{because} \ f \ \mathrm{is \ a \ homomorphism\ of\ partial\ group}\\
  \ & = \widetilde{f}(aK).\widetilde{f}(a'K) \ \ \ \mathrm{ by\ definition} \ of\ \ \widetilde{f}
\end{align*}

To show that $\widetilde{f}$ injects, it suffces to show that
$ker(\widetilde{f})$ is only the trivial element $\widetilde{e}$
of $G_{/K}$. Compute that if $\widetilde{f}(aK) = \widetilde{e}$
then $f(a) = \widetilde{e}$, and so $a \in K$. It follows by
Proposition ~\ref{Lad:011},(vi) that $aK = eK=\widetilde{e}$. The
map $\widetilde{f}$ surjects because $H = Im(f)$.
\end{proof}
\bigskip

\begin{proof}[Proof of Theorem~\ref{t:2}] Write $H=F.D'_{1}$ and $K=F'.D'_{2}$. By proposition ~\ref{Lad:011}.(i), we have
$HK=FF'$ and by  By proposition ~\ref{Lad:011}.(iv), we have
$HK_{/K}=FF'_{/F'}$ which is a group. It can be considered as a
partial group with support $E_{2}:=FF'_{/F'}$ and defect
$D_{2}:=\{e\}$. Routine verifications show that $HK$ is a partial
group having $K$
 as a normal partial subgroup and that $H\cap K$ is a normal partial subgroup of $H$.
 The map $f:\ H \longrightarrow HK_{/K}$ given by  $h \longmapsto hK$ is a surjective homomorphism having kernel $H\cap
 K$. In this case the default kernel of $f$ is equal its kernel
 since $D_{2}=\{e_{2}\}$.
  Therefore, the first theorem gives an isomorphism $H_{/(H \cap K)} \longrightarrow HK_{/K}$ given by
  $h(H\cap K)\longmapsto hK$.  The desired isomorphism is the inverse of the
isomorphism in the display.
\end{proof}

\medskip

\begin{proof}[Proof of Theorem~\ref{t:3}]  Write $G=E.D$, $K=F.D'$ and $N=F'.D''$. By proposition ~\ref{Lad:011}.(iv),
we have $G_{/N}=E_{/F'}$,  $K_{/N}=F_{/F'}$ and $G_{/K}=E_{/F}$
which are  groups. By applying the third theorem of group
isomorphism we have the map $$f:\ (E_{/F'})_{/(F_{/F'})}
\longrightarrow E_{/F}:\  \ xF'.(E_{/F'}) \longmapsto xF$$ is an
isomorphism. By Proposition ~\ref{Lad:011}.(ii), the map $f$ is
also defined as follow:

$$f:\ (G_{/N})_{/(K_{/N})} \longrightarrow G_{/K}:\  \ aN.(K_{/N})
\longmapsto aK.$$ The proof is completed
\end{proof}

\bibliographystyle{amsplain}
\vskip 0,4 cm

\end{document}